\documentclass[11pt,a4paper]{article}

\usepackage[math]{cellspace}
\usepackage{empheq}

\usepackage{psfrag}
\usepackage{amsmath,amssymb,cite}
\usepackage{latexsym}
\usepackage{graphicx}
\usepackage{lscape}
\usepackage{afterpage}
\usepackage[color=green!40]{todonotes}
\definecolor{OliveGreen}{rgb}{0,0.6,0}
\usepackage{hyperref}
\hypersetup{
	colorlinks=true,
	linkcolor=blue,
	filecolor=magenta,      
	urlcolor=cyan,
	citecolor=red,
}
\usepackage{cancel}

\DeclareMathOperator*{\argmin}{argmin}

\newcommand{\Id}{\ensuremath{\operatorname{Id}}}

\newcommand{\ds}{\displaystyle}
\newcommand{\nexto}{\kern -0.54em}
\newcommand{\dR}{{\rm {I\ \nexto R}}}

\newcommand{\dZ}{{\cal Z \kern -0.7em Z}}
\newcommand{\dC}{{\rm\hbox{C \kern-0.8em\raise0.2ex\hbox{\vrule
				height5.4pt width0.7pt}}}}
\newcommand{\dQ}{{\rm\hbox{Q \kern-0.85em\raise0.25ex\hbox{\vrule
				height5.4pt width0.7pt}}}}
\newcommand{\proofbox}{\hspace{\fill}{$\Box$}}
\newtheorem{lemma}{Lemma}
\newtheorem{theorem}{Theorem}
\newtheorem{corollary}{Corollary}
\newtheorem{proposition}{Proposition}

\newtheorem{fact}{Fact}
\newtheorem{remark}{Remark}

\newtheorem{algorithm}{Algorithm}
\newenvironment{proof}{Proof.}{\proofbox \\}

\providecommand{\norm}[1]{\lVert#1\rVert}
\newcommand{\dom}{\ensuremath{\operatorname{dom}}}

\newcommand{\prox}{\ensuremath{\operatorname{Prox}}}
\providecommand{\intr}{\ensuremath{\operatorname{int}}}
\providecommand{\fix}{\ensuremath{\operatorname{Fix}}}
\newcommand{\menge}[2]{\big\{{#1}~\big |~{#2}\big\}}
\newcommand{\scal}[2]{\left\langle{#1},{#2}  \right\rangle}
\newcommand{\sperp}{{\scriptscriptstyle\perp}}

\newcommand{\weakly}{\ensuremath{\:{\rightharpoonup}\:}}

\definecolor{myblue}{rgb}{0.9,0.9,0.98}

\usepackage[normalem]{ulem}
\newcommand{\stkout}[1]{\ifmmode\text{\sout{\ensuremath{#1}}}\else\sout{#1}\fi}
\definecolor{mypink}{rgb}{0.87, 0.19, 0.39}

\begin{document}
	
	\author{Authors}
	
	\author{
		Regina S. Burachik\thanks{Mathematics, UniSA STEM, University of South Australia, Mawson Lakes, S.A. 5095, Australia. E-mail: regina.burachik@unisa.edu.au, \,bethany.caldwell@mymail.unisa.edu.au, and yalcin.kaya@unisa.edu.au\,.}
		\and
Bethany I. Caldwell\footnotemark[1]
  \and
		C. Yal{\c c}{\i}n Kaya\footnotemark[1]
		\and
		Walaa M. Moursi\thanks{Department of Combinatorics and Optimization, University of Waterloo,
		Waterloo, Ontario N2L~3G1, Canada. E-mail: walaa.moursi@uwaterloo.ca\,.}
	}
	
	\title{\bf Optimal Control Duality and the Douglas--Rachford Algorithm}
	
	\maketitle
	
	\vspace*{-7mm}
	
	\begin{abstract} 
		{\noindent\sf  
		We explore the relationship between the dual of a weighted minimum-energy control problem, a special case of linear-quadratic optimal control problems, and the Douglas--Rachford (DR) algorithm.  We obtain an expression for the fixed point of the DR operator as applied to solving the optimal control problem, which in turn devises a certificate of optimality that can be employed for numerical verification. The fixed point and the optimality check are illustrated in two example optimal control problems. }\end{abstract}
	\begin{verse}		{\em Keywords}\/: {\sf Certification of optimality, Duality, 
   Douglas--Rachford algorithm,  Numerical methods, Optimal control, Double integrator, Machine tool manipulator}.
	\end{verse}
	\begin{verse}
	\end{verse}

	\pagestyle{myheadings}
	\markboth{}{\sf\scriptsize Optimal Control Duality and DR Algorithm by R. S. Burachik, B. I. Caldwell, C. Y. Kaya \& W. M. Moursi}

	\section{Introduction}
	
Linear--quadratic (LQ) optimal control problems constitute an important class encountered in many theoretical studies and areas of applications---see for example \cite{AltKaySch2016, BurKayMaj2014, ChrMauZir2010, AmmKen1998, MauObe2003, Mou2011, BusMau2000, KugPes1990}.  These problems are typically concerned with the minimization of a quadratic functional subject to linear differential equations and further affine constraints.  In this paper, we consider control-constrained {\em weighted minimum-energy control problems}\footnote{From a physics viewpoint, it is not neccessary to minimize the ``true'' energy of a dynamical system here. The main concern is rather to minimize the ``energy of the control or signal'' or the ``energy of the force.''  Elaborations of this subtle difference in terminology can also be found in \cite[Section~6.17]{AthFal1966}, \cite[Section~5.5]{Kirk1970}, \cite[Section~2.9]{Klamka2019} and \cite[page~118]{Sethi2019}.}, which are a special class of LQ optimal control problems.  

The Douglas--Rachford (DR) algorithm is an operator splitting method which has recently been applied to solving this special class of optimal control problems~\cite{BauBurKay2019,BurCalKay2022}.  In this paper, we explore the relationship between the dual of the optimal control problem and the DR algorithm.  In particular, we find an expression for the fixed point of the DR operator as applied to solving the optimal control problem (see Theorem \ref{thm:fixedpt}), which devises a certificate of optimality for a numerical solution.

A traditional approach to solving an LQ optimal control problem is to discretize the problem via some Runge--Kutta scheme and then apply a finite-dimensional large-scale optimization software, for example the AMPL--Ipopt suite~\cite{BurCalKay2022, AMPL, WacBie2006}.  The studies in~\cite{BauBurKay2019,BurCalKay2022} have shown that the application of the DR algorithm to the original infinite-dimensional problem (even for relatively simple instances) outperforms the traditional direct discretization approach.  Previously the DR algorithm has also been applied to solving discrete-time optimal control problems \cite{OdoStaBoy2013}; however, our main focus here will be the continuous-time (i.e., infinite dimensional) optimal control problem.

Duality theory for optimal control problems has been studied since the 1970\text{s} by Rockafellar~\cite{Rockafellar1970a, Rockafellar1971, Rockafellar1987}.  In particular,~\cite{Rockafellar1987} deals with general LQ control problems with state and control constraints.  Later \cite{HagIan1984} and \cite{BulKro2008}   used the classical Lagrangian function to derive the dual problem for optimal control problems.  Relatively recently the Fenchel dual of general LQ control problems has been derived in~\cite{BurKayMaj2014} in view of directly discretizing the dual problem and then applying the AMPL--Ipopt suite.  Most of our theoretical framework is similar to the duality approach in~\cite{BurKayMaj2014}, except that our formulation of the primal and dual problems is slightly modified so as to have primal and dual variables belonging to the same Hilbert space. 

To apply the DR splitting algorithm we write the primal problem as the problem of minimizing the sum of two convex functions.
The DR algorithm is employed to solve the monotone inclusion of finding  a 
zero of the sum of the {\em subdifferential} operators of these functions.
Of particular interest from a duality perspective is the fact that 
the DR splitting operator is self-dual, i.e.,
the splitting operator for the primal problem is the same as that for the dual problem (see \cite[Lemma~3.6 on page~133]{EckThesis}).

In the present paper, we consider the application of the DR algorithm to the dual of the control-constrained weighted minimum-energy control problem.  We derive an expression for the fixed point of the DR operator specific to optimal control (see Theorem \ref{thm:fixedpt}).  Then we use this expression in the verification of the optimality condition on the numerical solutions of two problems: one involving the double integrator, which is a simple but rich enough instance, and the machine tool manipulator, which is a more challenging instance. To the authors' best knowledge this interplay between the DR algorithm and duality of (infinite dimensional) optimal control problems has not been previously explored. 


The paper is organized as follows. In Section~2 we provide the preliminaries, where we introduce the mathematical model of the optimal control problem, split the constraints into an affine set and a box, and prove results about the projection onto the affine set.  We also present in this section the optimality conditions for the control problem.  In Section~3, we introduce the dual of the optimal control problem and transform it into a new form suitable for our remaining analysis.  We derive the proximity operators and deduce that the new form is the Fenchel dual of the primal problem.  In Section~4, we introduce the DR operator, derive its fixed point, and provide the algorithm we propose to use for the optimal control problem. The latter algorithm generates both primal and dual sequences.  In Section~5 we perform computations to illustrate the algorithm and the convergence of the primal and dual iterates, via problems involving the double integrator and a machine tool manipulator.  Furthermore we verify the optimality conditions using the certificate we devised in Section~4, for the same problems.  Finally, in Section~6 we provide some concluding remarks.

\section{Preliminaries}

In the weighted minimum-energy control problem, the aim is to find a control $u$ which {\em minimizes} the {\em quadratic objective functional}
 	\begin{equation}  \label{quad_obj_fun}
		\ds \frac{1}{2}\int_{0}^{1} r(t)\,u^2(t)\,dt\,,
	\end{equation}
subject to the linear differential equation constraints
 	\begin{equation}  \label{ODE}
		\dot{x}(t) = A(t)\,x(t) + b^T(t)\,u(t)\,,\ \ \mbox{for a.e.\ } t\in[0,1]\,,
	\end{equation}
	with $\dot{x} := dx/dt$, and the boundary conditions
	\begin{equation}  \label{BC}
		\psi(x(0),x(1)) = 0\,.
	\end{equation}
	We define the {\em state variable vector} $x:[0,1]\to\dR^n$ $x(t) := (x_1(t),\ldots,x_n(t))\in\dR^n$
 and the {\em control variable} 
$u:[0,1]\to\dR$ with $u(t)\in\dR$. 
The time-varying matrices $A:[0,1]\to \dR^{n\times n}$ and 
$b:[0,1]\to\dR^{n\times 1}$
are continuous, and $b$ is not the zero vector. We also assume that 
$r:[0,1]\to \dR_{++}$ is continuous.
The vector function $\psi:\dR^{2n}\to\dR^s$, with $\psi(x(0),x(1)) := (\psi_1(x(0),x(1)),\ldots,\psi_s(x(0),x(1)))\in\dR^s$, is affine. Without loss of generality the time horizon in~\eqref{quad_obj_fun}--\eqref{BC} is set to be $[0,1]$ unless stated otherwise.
 	
	Although a vast majority of the studies on LQ control in the optimal control literature deal with the above problem with no constraints imposed on the control variable $u$, it is much more realistic, especially in practical situations, to consider restrictions on the values that $u$ is allowed to take.  In many applications, it is common practice to impose simple bounds on the components of $u(t)$; namely,
 \begin{equation}  \label{bounds}
		\underline{a}(t) \le u(t) \le \overline{a}(t)\,,\ \ \mbox{for a.e.\ } t\in[0,1]\,,
	\end{equation}
	where, respectively, the lower and upper bound functions 
  $\underline{a},\overline{a}:[0,1]\to\dR$ are continuous and $\underline{a}(t) < \overline{a}(t)$, for all $t\in[0,1]$.
In other words, we {\em formally} state
	\[
		u(t)\in U(t) := [\underline{a}(t),\overline{a}(t)]\subset \dR,\ \ \mbox{for a.e.\ } t\in[0,1]\,,
	\]
	as an expression alternative but equivalent to~\eqref{bounds}.
 
	The objective functional in~\eqref{quad_obj_fun} and the constraints in~\eqref{ODE}--\eqref{bounds} can be put together to present the control-constrained weighted minimum-energy control problem as follows.	\[
	\mbox{(P) }\left\{\begin{array}{rl}
		\ds\min_{u(\cdot)} & \ \ \ds\frac{1}{2}\int_{0}^{1} r(t)u^2(t)\,dt \\[5mm] 
		\mbox{subject to} & \ \ \dot{x}(t) = A(t)x(t) + b^T(t)u(t)\,,\ \ \mbox{for a.e.\ } t\in[0,1]\,, \\[2mm]
		& \ \ \psi(x(0),x(1)) = 0\,, \\[2mm]
		& \ \ \underline{a}(t) \le u(t) \le \overline{a}(t)\,,\ \ \mbox{for a.e.\ } t\in[0,1]\,.
	\end{array} \right.
	\]
We pose the primal variable in (P) as $u$, since every given $u$ generates a unique $x$ via the ODE system. 

In Problem~(P), the control variable $u$ can in general be a vector, namely $u:[0,1]\to\dR^m$ with $m$ components, $u(t) = (u_1(t),\ldots,u_m(t))\in\dR^m$, with lower and upper bounds imposed on each of the $m$ control variables. For clarity and neatness of the expressions, we only consider a single (or scalar) control variable $u$ (for $m = 1$).  Otherwise, the results in this paper easily extend to the case of multiple control variables, thanks to the separability of the projections.

\subsection{Constraint splitting}

	We split the constraints of Problem~(P) into two sets:
\begin{eqnarray} 
		&& {\cal A} := \big\{u\in 
        L^{2}([0,1];\dR)
  \ |\ \exists x\in W^{1,2}([0,1];\dR^n)\mbox{ which solves } \nonumber \\[1mm]
		&&\hspace*{45mm} \dot{x}(t) = A(t)x(t) +  b^T(t)u(t)
  \,,\ \ \mbox{for a.e.\ } t\in[0,1]\,, \mbox{ and }\nonumber \\[1mm]
		&&\hspace*{45mm} \psi(x(0),x(1)) = 0 \big\}\,, \nonumber \\[2mm]
		&& {\cal B} := \big\{u\in L^{2}([0,1];\dR)\ |\ \underline{a}(t) \le u(t) \le \overline{a}(t)\,,\ \mbox{for a.e.\ } t\in[0,1]\big\}\,, \nonumber
	\end{eqnarray}
where $W^{1,2}([0,1];\dR^n)$ is the Sobolev space of absolutely continuous functions, namely,
\[
{W}^{1,2}([0,1];\dR^n)=\left\{ z\in {L}^2([0,1];\dR^n) \,\,|\,\,
\dot z = dz/dt \in {L}^2([0,1];\dR^n)\,\right\}.
\]
We assume that the control system  
$\dot{x}(t) = A(t)x(t) + b^T(t)u(t)$
is {\em controllable}~\cite{Rugh1996}. The latter means that there exists a (possibly not unique) $u(\cdot)$ such that, when this $u(\cdot)$ is substituted, the boundary-value problem given in ${\cal A}$ has a solution $x(\cdot)$.  In other words, controllability is equivalent to ${\cal A} \neq \emptyset$.  Also, clearly, ${\cal B} \neq \emptyset$.  We observe that the constraint set~${\cal A}$ is an {\em affine subspace} and ${\cal B}$ a {\em box}, constituting two convex sets in a Hilbert space.  In particular, we note that ${\cal B}$ 
is closed in $L^2([0,1];\dR)$. 

		In the following we set
 \begin{equation}
a^\sperp=P_{\cal A} 0,\label{eq:a}
\end{equation}
where $P_{\cal C}$ is the {\em orthogonal projection} onto the nonempty closed and convex set $\cal C$.
	
We now prove the following useful lemma
which we shall use in the sequel.
\begin{lemma}
\label{lem:aperp}
The following hold:
\begin{enumerate}
\item
\label{lem:aperp:i}
${\cal A}=a^\sperp+(\cal A-\cal A)$.
\item	
\label{lem:aperp:ii}
$a^\sperp\in (\cal A-\cal A)^\perp$.
\item
\label{lem:aperp:iii}
$P_{\cal A}=a^\sperp+P_{\cal A-\cal A}$.

\end{enumerate}		
\end{lemma}	
\begin{proof}
\eqref{lem:aperp:i}:
Because $\cal A$ is an affine subspace, 
we can simply write it as ${\cal A}=s+(\cal A-\cal A)$ for any $s\in \cal A$ and observe 
that $\cal A-\cal A$ is a linear subspace. 
In particular, we can set $s:=a^\sperp\in \cal A$.	

\eqref{lem:aperp:ii}:
To see this we note that  
$a^\sperp=P_{\cal A}(0)= P_{a^\sperp+(\cal A-\cal A)}(0)
=a^\sperp\,\,+P_{\cal A-\cal A}(0-a^\sperp)=a^\sperp-P_{\cal A-\cal A}(a^\sperp)
=(\Id-P_{\cal A-\cal A})(a^\sperp)=P_{(\cal A-\cal A)^\perp}(a^\sperp)$, where we use part \ref{lem:aperp:i}, the translation formula \cite[Proposition~3.19]{BC2017}, the linearity of $P_{\cal A-\cal A}$, and the fact that $\Id-P_{\cal A-\cal A}=P_{(\cal A-\cal A)^\perp}$
(see \cite[Corollary~3.24(iii)~and~(v)]{BC2017}).

\eqref{lem:aperp:iii}:
Indeed, we have
$P_{\cal A}= P_{a^\sperp+(\cal A-\cal A)}
=a^\sperp\,\,+P_{\cal A-\cal A}(\cdot-a^\sperp)$
$=a^\sperp+P_{\cal A-\cal A}(\cdot)-P_{\cal A-\cal A}(a^\sperp)
=(\Id-P_{\cal A-\cal A})(a^\sperp)+P_{\cal A-\cal A}(\cdot) =P_{(\cal A-\cal A)^\perp}(a^\sperp)+P_{\cal A-\cal A}(\cdot)=a^\sperp +P_{\cal A-\cal A}(\cdot) $, 
where, besides from part \ref{lem:aperp:ii}, we used again the translation formula \cite[Proposition~3.19]{BC2017},
the linearity of $P_{\cal A-\cal A}$, and the fact that $\Id-P_{\cal A-\cal A}=P_{(\cal A-\cal A)^\perp}$. 	
\end{proof}

	\subsection{Optimality conditions}
	\label{sec:optimality}
	
	In what follows we will derive the necessary conditions of optimality for Problem~(P), using the {\em maximum principle}.  Various forms of the maximum principle and their proofs can be found in a number of reference books---see, for example, \cite[Theorem~1]{PonBolGamMis1986}, \cite[Chapter~7]{Hestenes1966}, \cite[Theorem~6.4.1]{Vinter2000}, \cite[Theorem~6.37]{Mordukhovich2006}, and \cite[Theorem~22.2]{Clarke2013}.  We will state the maximum principle suitably utilizing these references for our setting and notation.  
	
	First, define the {\em Hamiltonian function} 
$H:\dR^n \times \dR \times \dR^n \times [0,1]\to \dR$
 for Problem~(P) as
 	\[
		H(x,u,\lambda,t) := \frac{1}{2} r(t)\,u^2 + \left\langle\lambda, A(t)\,x + b^T(t)\,u \right\rangle\,,
	\]
	where $\lambda(t) := (\lambda_1(t),\ldots,\lambda_n(t))\in\dR^n$ is the {\em adjoint variable} (or {\em costate}) {\em vector} such that
	\[ 
	\dot{\lambda}(t) := -\frac{\partial H}{\partial x}(x(t), u(t), \lambda(t), t)\,,
	\] 
	i.e.,
	\begin{equation} \label{eq:adjoint} 
		\dot{\lambda}(t) = -A^T(t)\,\lambda(t)\,,
	\end{equation} 
	where the {\em transversality conditions} involving $\lambda(0)$ and $\lambda(1)$ depend on the boundary condition $\psi(x(0),x(1)) = 0$, but are not needed for our purposes and therefore omitted.
	
	\noindent 
	{\bf Maximum Principle.}\ \  Suppose that $u\in L^2([0,1];\dR)$
 is optimal for Problem~(P).  Then there exists a continuous adjoint variable vector $\lambda\in W^{1,2}([0,1];\dR^n)$ as defined in~\eqref{eq:adjoint}, such that $\lambda(t)\neq{\bf 0}$ for any $t\in[0,1]$, and that, for a.e.\ $t\in[0,1]$,
 \begin{equation}  \label{cond:opt_u}
		u(t) = \argmin_{w\in[\underline{a}(t), \overline{a}(t)]} H(x(t), w, \lambda(t), t) = \argmin_{w\in[\underline{a}(t), \overline{a}(t)]} \frac{r(t)}{2}\,w^2 + b^T(t)\,\lambda(t)\, w\,.
	\end{equation}
Condition~\eqref{cond:opt_u} in turn yields the optimal control as
 	\begin{equation}  \label{opt_u(t)}
		u(t) = \left\{\begin{array}{ll}
			\overline{a}(t)\,, &\ \ \mbox{if\ \ } -b^T(t)\,\lambda(t) > r(t)\,\overline{a}(t)\,, \\[1mm]
			-b^T(t)\,\lambda(t)/r(t)\,, &\ \ \mbox{if\ \ } r(t)\,\underline{a}(t) \le -b^T(t)\,\lambda(t) \le r(t)\,\overline{a}(t)\,, \\[1mm]
			\underline{a}(t)\,, &\ \ \mbox{if\ \ } -b^T(t)\,\lambda(t) < r(t)\,\underline{a}(t)\,,
		\end{array} \right.
	\end{equation}
	for all\ $t\in[0,1]$.  

	\section{Reformulation of the Dual Problem}
	\label{DR_duality}	

  In what follows, we 
  suppress/omit the dependence on $t$ of the specified data in the problem whenever it is convenient for clarity.  For example we write $r(t)$ as $r$, $A(t)$ as $A$, and so on.

 \subsection{Dual Problem}
	\label{sec:duality}
	
	The dual of a control-constrained LQ control problem was first given in~\cite{BurKayMaj2014} for a single control variable.  Then a generalization to multiple control variables was carried out in a straightforward manner in~\cite{AltKaySch2016}.  For simplicity of exposition, suppose that the boundary condition vector $\psi(x(0),x(1)) = 0$ is given as
	\begin{equation}
     \label{eq:BC}
     	x(0) = x_0\quad\mbox{and}\quad x(1) = x_f\,,
 \end{equation}
where $x_0,x_f\in\dR^n$.  All theoretical results below can be easily extended to the case of a general affine function $\psi$; for example as in~\cite{AltKaySch2016}.  Now, using~\cite{BurKayMaj2014}, the dual of Problem~(P) can subsequently be presented as follows.
	\[
	\mbox{(D1) }\left\{\begin{array}{rl}
		\ds\min_{p(\cdot)} & \ \ \ds\int_0^1 \theta(p(t),t)\, dt - \left(x_f^T\,p(1) - x_0^T\,p(0)\right)\\[5mm]
		\mbox{subject to} & \ \ \dot{p}(t) = -A^T\,p(t)\,,\ \ \mbox{for all\ } t\in[0,1]\,,
	\end{array} \right.
	\]
	where
	\begin{equation} \label{vartheta_p}
		\theta(p,t) := \left\{\begin{array}{ll}
			\overline{a}\,b^T\,p - r\,\overline{a}^2 / 2\,, 
			& \mbox{ if\ \ } b^T\,p >  r\,\overline{a}\,, \\[2mm]
			(b^T\,p)^2 / (2\,r)\,,
			& \mbox{ if\ \ } r\,\underline{a} \le b^T\,p
			\le  r\,\overline{a}\,, \\[2mm]
			\underline{a}\,b^T\,p - r\,\underline{a}^2 / 2\,, 
			& \mbox{ if\ \ } b^T\,p <  r\,\underline{a}\,,
		\end{array}\right.
	\end{equation}
	for all $t\in[0,1]$.  In the case of multiple controls, $\theta(p(t),t)$ in the dual objective functional is replaced by $\sum_{i=1}^m \theta_i(p(t),t)$, $\theta(p,t)$ in \eqref{vartheta_p} by $\theta_i(p,t)$,
	and $(\overline{a},\underline{a}, b, r)$
	by $(\overline{a}_i,\underline{a}_i, b_i, r_i)$,
	$i = 1,\ldots,m$. 
	
	We note that $p\in W^{1,2}([0,1];\dR^n)$ is the optimization variable of Problem~(D1).  We recall that the saddle-point property and the strong duality results given in~\cite[Theorem~2]{BurKayMaj2014}, as well as the hypothesis in the same theorem, imply that $p = -\lambda$,  where $\lambda(\cdot)$ is the adjoint variable of Problem~(P) satisfying~\eqref{eq:adjoint}.
  
For the analysis of the dual problem (D1), we need the gradient of $\theta$, which we consider next.

 \begin{remark}\label{rem:gradient of theta} \rm
 Let the gradient of $\theta$ w.r.t.~$p$ be denoted as $\nabla_p \theta := \partial\theta / \partial p$.  The definition~\eqref{vartheta_p} directly yields
	\begin{equation} \label{grad_vartheta}
		\nabla_p\theta(p,t) = \left\{\begin{array}{ll}
			\overline{a}\,b\,, 
			& \mbox{ if\ \ } b^T\,p >  r\,\overline{a}\,, \\[2mm]
			b\,b^T\,p / r\,,
			& \mbox{ if\ \ } r\,\underline{a} \le b^T\,p
			\le  r\,\overline{a}\,, \\[2mm]
			\underline{a}\,b\,, 
			& \mbox{ if\ \ } b^T\,p <  r\,\underline{a}\,.
		\end{array}\right.
	\end{equation}
	By comparing \eqref{opt_u(t)} and \eqref{grad_vartheta}, and using the fact that $p = -\lambda$, one obtains a relationship with the control variable of Problem~(P) as follows.
	\begin{equation} \label{biui}
		\nabla_p\theta(p(t),t) = \nabla_p\theta(-\lambda(t),t) = b\,u(t)\,.
	\end{equation}
 \proofbox
 \end{remark}
	
	The objective functional in Problem~(D1) is in the so-called {\em Bolza form}, which contains both an integral term and a term involving endpoints, and can be converted into the {\em Lagrange form}, which contains only an integral term.  In the following proposition, we convert the initial and terminal costs in the objective function into the Lagrange form, by using the differential equations for $x$ and $p$ in Problems~(P) and (D1), respectively.
			
 \begin{proposition}
 Consider the notation of problem (D1).  Fix $u\in \mathcal{A}$ and take
 $x$ as the corresponding solution of the ODE system in $\mathcal{A}$ with boundary conditions \eqref{eq:BC}. Let $p$ be such that it verifies the constraints of (D1) (i.e., $\dot{p} = -A^T\,p$). Then,
\begin{equation}  
\label{e:perpscal}
x_f^T\,p(1) - x_0^T\,p(0)  =\langle u, b^Tp\rangle\,.
\end{equation}
In particular, we have that
\begin{equation}  
\label{e:perpscal2}
x_f^T\,p(1) - x_0^T\,p(0) 
 =\langle a^\perp, b^Tp\rangle\,,
\end{equation}
for every $p$ s.t. $\dot{p} = -A^T\,p$. Consequently, the objective functional of (D1) can be rewritten as
\begin{equation} 
\label{Lagr_form}
		\int_0^1 \left[\theta(p(t),t)\, 
		- \langle \nabla_p\theta(p(t),t), p\rangle \right]\,dt\,,
\end{equation}
where
\begin{equation} \label{grad_vartheta_p}
		\langle \nabla_p\theta(p,t), p \rangle = \left\{\begin{array}{ll}
			\overline{a}\,b^T\,p\,, 
			& \mbox{ if\ \ } b^T\,p >  r\,\overline{a}\,, \\[2mm]
			(b^T\,p)^2 / r\,,
			& \mbox{ if\ \ } r\,\underline{a} \le b^T\,p
			\le  r\,\overline{a}\,, \\[2mm]
			\underline{a}\,b^T\,p\,, 
			& \mbox{ if\ \ } b^T\,p < r\,\underline{a}\,.
\end{array}\right.
\end{equation}
\end{proposition}
 
\begin{proof} Equation \eqref{e:perpscal2}  will follow from \eqref{e:perpscal} and the fact that, by its definition (see \eqref{eq:a}), $a^\perp \in \mathcal{A}$. Thus, we proceed now to establish \eqref{e:perpscal}. Indeed, using \eqref{eq:BC} and the Fundamental Theorem of Calculus gives
	\begin{align*}
		x_f^T\,p(1) - x_0^T\,p(0) &= x^T(1)\,p(1) - x^T(0)\,p(0)
		= \int_0^1 \frac{d}{dt}\langle x(t), p(t) \rangle\, dt 
		\nonumber
		\\
		&= \int_0^1 \big( \langle \dot{x}(t), p(t) \rangle
		+ \langle x(t), \dot{p}(t) \rangle \big)\,dt 
		\nonumber
		\\
		&= \int_0^1 \big( \langle A(t)\,x(t) + b(t)\,u(t), p(t) \rangle
		+ \langle x(t), -A^T(t)\, p(t) \rangle \big)\,dt 
		\nonumber
		\\
		&= \int_0^1 \big( \langle A(t)\,x(t), p(t) \rangle 
		+ \langle b(t)\,u(t), p(t) \rangle
		- \langle A(t)\,x(t), p(t) \rangle \big)\,dt 
		\nonumber
		\\
		&= \int_0^1 \langle b(t)\,u(t), p(t) \rangle\, dt\,=  \langle b\,u, p\rangle=\langle u, b^Tp\rangle\,,
	\end{align*}
 where we used the fact that $u\in \mathcal{A}$, the definition of $x$, and the fact that $p$ verifies the constraints of (D1). This proves \eqref{e:perpscal}.
Now, using \eqref{biui}, the objective functional of (D1) can equivalently be written in the Lagrange form as in 
\eqref{Lagr_form}.
Finally \eqref{grad_vartheta_p} 
follows from~\eqref{grad_vartheta}.
\end{proof}

Next we collect the previous results to derive a simple form for the dual, where the min in (D1) was replaced by max in order to avoid negative signs in the objective function. 

 \begin{corollary}
 Problem~(D1) in the so-called Lagrange form is
	\[
	\mbox{(D1neat) }\left\{\begin{array}{rl}
		\ds\max_{p(\cdot)} &\ \ \ds\int_0^1 \phi(p(t),t)\,dt  \\[5mm]
		\mbox{subject to} &\ \ \dot{p}(t) = -A^T\,p(t)\,,\ \ \mbox{for all\ } t\in[0,1]\,,
	\end{array} \right.
	\]
	where
	\[
		\phi(p,t) := \left\{\begin{array}{ll}
			r\,\overline{a}^2 / 2\,, 
			& \mbox{ if\ \ } b^T\,p >  r\,\overline{a}\,, \\[2mm]
			(b^T\,p)^2 / (2\,r)\,,
			& \mbox{ if\ \ } r\,\underline{a} \le b^T\,p \le r\,\overline{a}\,, \\[2mm]
			r\,\underline{a}^2 / 2\,, 
			& \mbox{ if\ \ } b^T\,p < r\,\underline{a}\,.
		\end{array}\right.
	\]
\end{corollary}	
 \begin{proof}
This follows from the substitution of \eqref{vartheta_p} and \eqref{grad_vartheta_p} into \eqref{Lagr_form} and subsequent cancellations.  
 \end{proof}
 
	The optimization variable $p(\cdot)$ of Problems (D1) or (D1neat) is not the {\em dual variable} per se since it does not live in the same space as $u$.  We propose as the dual variable $w\in L^2([0,1];\dR)$ (in the same space as $u$) such that
	\begin{equation} \label{eq:w}
		w := b^T\,p\,.
	\end{equation}
 \begin{corollary}
	We can re-write the dual problem~(D1) as
	\[
	\mbox{(D) }\left\{\begin{array}{rl}
		\ds\min_{w(\cdot)} & \ \ \ds\int_0^1 \vartheta(w(t),t)\, dt - \left(x_f^T\,p(1) - x_0^T\,p(0)\right) \\[5mm]
		\mbox{subject to} & \ \ \dot{p}(t) = -A^T(t)\,p(t)\,,\quad w(t) = b^T(t)\,p(t)\,,\ \ \mbox{for all\ } t\in[0,1]\,,
	\end{array} \right.
	\]
	where, after omitting dependence on $t$ again for clarity in appearance, gives
	\begin{equation} \label{vartheta_w}
		\vartheta(w,t) := \left\{\begin{array}{ll}
			\overline{a}\,w - r\,\overline{a}^2 / 2\,, 
			& \mbox{ if\ \ } w >  r\,\overline{a}\,, \\[2mm]
			w^2 / (2\,r)\,,
			& \mbox{ if\ \ } r\,\underline{a} \le w
			\le  r\,\overline{a}\,, \\[2mm]
			\underline{a}\,w - r\,\underline{a}^2 / 2\,, 
			& \mbox{ if\ \ } w <  r\,\underline{a}\,.
		\end{array}\right.
	\end{equation}
\end{corollary}	
\begin{proof}
   Substitution of \eqref{eq:w} into (D1) furnishes the corollary. 
\end{proof}

	\subsection{Proximity operators and verification of (D) as the dual of (P)}

 Let ${\cal C}$ be a nonempty closed convex subset of $L^2([0,1];\dR)$.
 Recall that $\iota_{\cal C}$ is the {\em indicator function} of ${\cal C}$ given by
 \[
\iota_{\cal C}(x) := \left\{\begin{array}{ll}
    0\,, & \mbox{ if\ \ } x\in{\cal C}\,, \\
    \infty\,, & \mbox{ otherwise},
    \end{array}\right.
 \]
and the {\em normal cone to ${\cal C}$} is given by $N_{\cal C} :=\partial\iota_{\cal C}$, the {\em subdifferential of} $\iota_{\cal C}$. The shortest distance from a point $y\in {\cal H}$ to the set ${\cal C}$ is given by $d_{\cal C}(y):=\| y-P_{\cal C}(y)\|$. 
	
Observe that problem~(P) can be written in a concise form as 
	\begin{equation}  
	\label{eq:problem:min:norm}
		\min_{u\in L^2([0,1];\dR)}  f(u) + g(u)\,,
	\end{equation}
	where 
	\begin{equation}  \label{eq:fg}
		f=\frac{r}{2}\norm{\cdot}_{L^2}^2+\iota_{\cal B}\quad\mbox{and}\quad g=\iota_{\cal A}\,,
	\end{equation}
	Let $f^*$ (respectively $g^*$) denote the {\em Fenchel conjugate} of $f$ (respectively $g$), defined by
	\begin{equation}
	    f^*(w)=\sup_{u\in L^2([0,1];\dR)}(\scal{u}{w}-f(u)).
	\end{equation}

	Recall that the Fenchel  dual of Problem~(P) is (see, e.g., \cite[Definition~15.10]{BC2017})
	\begin{equation}  \label{eq:dualproblem:min:norm}
		\min_{w\in L^2([0,1];\dR)}  f^*(w) + g^*(-w),
	\end{equation}
	where 
	\begin{equation}  
	\label{eq:f*g*}
		f^*(v) = \tfrac{1}{2r} \norm{v}_{L^2}^2 -\tfrac{r}{2} d_{\cal B}^2(v/r)\mbox{, }\quad g^* = \iota_{(\cal A-\cal A)^\perp} + \scal{\cdot}{a^\sperp},
	\end{equation}
	$a^\sperp = P_{\cal A}0$. The formula for $f^*$ can be deduced from \cite[Examples 12.21 and 13.4]{BC2017}, while the formula for $g^*$ can be deduced from \cite[Example 13.3(iii) and Proposition~13.23(iii)]{BC2017}.

	Recall that the {\em proximity operator}, or proximity mapping, of a functional $h$ is defined by \cite[Definition~12.23]{BC2017}:
	\begin{equation}  \label{def:prox}
		\prox_h(u) := \argmin_{y\in L^2([0,1];\dR)}
		\left(h(y) + \frac{1}{2}\norm{y - u}_{L^2}^2  \right)
	\end{equation}
	for any $u\in L^2([0,1];\dR$).

The next lemma extends \cite[Proposition~2.1]{BauBurKay2019}. The quoted proposition addresses the particular case of the double integrator, i.e., when $n:=2$, and the ODE system has $A:=\begin{bsmallmatrix} 0 & 1 \\ 0 & 0 \end{bsmallmatrix}$ and $b:=\begin{bsmallmatrix} 0 \\ 1 \end{bsmallmatrix}$. 
 
	\begin{lemma} \label{lem:proxfg}
		Suppose that Problem~(P) is written in the form in~\eqref{eq:problem:min:norm}. Then
		\begin{equation}  \label{eq:proxf}
			\prox_f(u)(t) = \left\{\begin{array}{ll}
				\overline{a}\,, & \mbox{ if\ \ } u(t) > (r+1)\,\overline{a}\,, \\[2mm]
				u(t) / (r+1)\,, & \mbox{ if\ \ } (r+1)\,\underline{a} \le u(t)
				\le (r+1)\,\overline{a}\,, \\[2mm]
				\underline{a}\,, & \mbox{ if\ \ } u(t) < (r+1)\,\underline{a}\,;
			\end{array}\right.
		\end{equation}
		and
		\begin{equation}  \label{eq:proxg}
			\prox_g(u)(t) = -b^T\lambda(t) + u(t)\,,
		\end{equation}
		where $\lambda$ solves $\dot{\lambda}(t) = -A^T\lambda(t)$.
	\end{lemma}
	\begin{proof}\ \\
		Using the functional $f$ in~\eqref{eq:fg} and the definition in~\eqref{def:prox},
		\[
		\prox_f(u) = \argmin_{y\in L^2([0,1];\dR)}    \left(\frac{r}{2}\norm{y}_{L^2}^2+\iota_{\cal B}(y) + \frac{1}{2}\norm{y - u}_{L^2}^2  \right)\,.
		\]
		In other words, finding $\prox_f(u)$ is finding $y$ that solves the problem
		\[
		\mbox{(Pf) }\left\{\begin{array}{rl}
			\ds\min_{y(\cdot)} & \ \ \ds\frac{1}{2}\int_0^1 \left(r(t)\,y^2(t) + \left(y(t) - u(t)\right)^2\right)\, dt \\[5mm]
			\mbox{subject to} & \ \ y\in{\cal B}\,.
		\end{array} \right.
		\]
		The solution to Problem~(Pf) is simply given by
		\[
		y(t) = \argmin_{\underline{a} \le v \le \overline{a}} \big(\, r\,v^2 + (v - u(t))^2\,\big)\,,
		\]
		which then, after straightforward manipulations, yields~\eqref{eq:proxf}.  The proximity operator of $g$ can  similarly be computed, using the functional $g$ in~\eqref{eq:fg} and \eqref{def:prox}:
		\[
		\prox_g(u) = \argmin_{y\in L^2([0,1];\dR)} \left(\iota_{\cal A}(y) + \frac{1}{2}\norm{y - u}_{L^2}^2  \right)\,,
		\]
		for any $u\in L^2([0,1];\dR)$.  In this case, finding $\prox_g(u)$ is finding $y$ that solves the problem
		\[
		\mbox{(Pg) }\left\{\begin{array}{rl}
			\ds\min_{y(\cdot)} & \ \ \ds\frac{1}{2}\int_0^1 \left(y(t) - u(t)\right)^2\, dt \\[5mm]
			\mbox{subject to} & \ \ \dot{x}(t) = A\,x(t) + b\,y(t)\,,\ \ x(0) = x_0\,,\ \ x(1) = x_f\,.
		\end{array} \right.
		\]
		Problem~(Pg) is a classical optimal control problem with $x$ the state variable vector and $u$ the scalar control variable.  Define the Hamiltonian function:
		\[
		H(x,y,\lambda,t) := \frac{1}{2} \left(y - u(t)\right)^2 + \lambda^T\left(A\,x(t) + b\,y(t)\right)\,,
		\]
		where $\lambda$ is the adjoint variable, defined as in Section~\ref{sec:optimality} as $\dot{\lambda}(t) = -A^T\lambda(t)$.  The necessary and sufficient condition of optimality for Problem~(Pg) is then given by
		\[
		\frac{\partial H}{\partial y}(x,y,\lambda,t) = y - u(t) + b^T\lambda = 0\,,
		\]
		which, when solved for $y$, yields~\eqref{eq:proxg}.
	\end{proof}
	
	\begin{remark} \label{rem:projA} \rm
		We note that Problem~(Pg) is nothing but the problem of finding a projection of $u$ onto ${\cal A}$; namely that $\prox_g(u) = {\cal P}_{\cal A}(u)$.  We also note from~\eqref{eq:proxf} that 
		\begin{equation}  \label{proj_ball}
		    \prox_f(u) = {\cal P}_{\cal B}(u/(r+1))\,.
		\end{equation}  
		Therefore, with $r = 0$, one recovers the projection onto the $L^\infty$-ball; namely,
		\begin{equation}
		\label{e:proj:ball}
		 {\cal P}_{\cal B}(u)(t)  
		  =
		\left\{\begin{array}{ll}
				\overline{a}\,, & \mbox{ if\ \ } u(t) > \,\overline{a}\,, \\[2mm]
				u(t) \,, & \mbox{ if\ \ } \,\underline{a} \le u(t)
				\le \,\overline{a}\,, \\[2mm]
				\underline{a}\,, & \mbox{ if\ \ } u(t) < \,\underline{a}\,.
			\end{array}\right.  
		\end{equation}
		  We also observe that $\prox_f(u)$ is piecewise-$C^1$, namely piecewise-linear and continuous, in $u$.
		\proofbox
	\end{remark}

We start by providing an alternative formula for $f^*$ which will be useful in the sequel.
	\begin{lemma}
		\label{lem:conj:f}
		We have 
		\begin{equation} \label{eq:f*}
			{f^*}(w) = \int_0^1 \vartheta(w,t)\,dt\,,
		\end{equation}
		where $\vartheta$ is defined as in~\eqref{vartheta_w}.
	\end{lemma}
	
	\begin{proof}
		It follows from
		\eqref{eq:f*g*} that 
		\begin{subequations}
			\begin{align}
				f^*(w)
				&=\tfrac{1}{2r}\norm{w}^2-\tfrac{r}{2}d^2_{\cal B}(w/r)
				\\
				&=\tfrac{1}{2r}\norm{w}^2-\tfrac{r}{2}\norm{w/r-\mathcal{P}_{\cal B}(w/r)}^2
				\\
				&=\tfrac{1}{2r}\norm{w}^2-\tfrac{r}{2}(\norm{w/r}^2+\norm{\mathcal{P}_{\cal B}(w/r)}^2-2\scal{w/r}{\mathcal{P}_{\cal B}(w/r)})
				\\
				&=\tfrac{r}{2}(2\scal{w/r}{\mathcal{P}_{\cal B}(w/r)}-\norm{\mathcal{P}_{\cal B}(w/r)}^2)
				\\
				&=\scal{w}{\mathcal{P}_{\cal B}(w/r)}-\tfrac{r}{2}\norm{\mathcal{P}_{\cal B}(w/r)}^2,
			\end{align}	
		\end{subequations}
		and the conclusion follows in view of \eqref{e:proj:ball} and \eqref{vartheta_w}.
	\end{proof}
 
In the next result, 
we provide a concrete formula for $g^*$.
As a byproduct, we obtain a formula 
for $(\cal A-\cal A)^\perp$.	
	
	\begin{theorem}
		\label{t:calc:g^*}
		Set 
		\begin{eqnarray} 
			S &:=& \big\{w\in L^{2}([0,1];\dR)\ |\ \, w(t) = b^T(t)\,p(t)\,, \mbox{ for all\ } t\in[0,1]\,, \nonumber \\[1mm] 
			&&\hspace{15mm} \mbox{ where } p \in W^{1,2}([0,1];\dR^n)\mbox{ solves } \dot{p}(t) = -A^T(t)\,p(t) \big\},
			\label{eq:sperp}
		\end{eqnarray}
		and set 
		\begin{equation} \label{e:g*}
			\widetilde{g}(w) = \iota_{S}(w) + \left(x_f^T\,p(1) - x_0^T\,p(0)\right),
		\end{equation}
		where $p$ solves $\dot{p}(t) = -A^T(t)\,p(t)$, and $w(0) = b^T(0)\,p(0)$ and $w(1) = b^T(1)\,p(1)$.  Then 
		the following hold:
		\begin{enumerate}
			\item
   \label{t:calc:g^*:ii}
			$(\forall w \in S)$ 
			we have 
			\begin{equation}\label{eq:aperp3}
				\scal{a^\sperp}{w}=x_f^T\,p(1) - x_0^T\,p(0)\,.
			\end{equation}	
			
			\item
			\label{t:calc:g^*:i}
			$S=(\cal A-\cal A)^\perp$.
			\item
			\label{t:calc:g^*:iii}
			$g^*=\widetilde{g}$.
		\end{enumerate}	
	\end{theorem}
	\begin{proof} 
		\eqref{t:calc:g^*:ii}:
  Note that \eqref{eq:aperp3} follows directly from \eqref{e:perpscal2} and the definition of $w$. \\
  	\eqref{t:calc:g^*:i}:
		Using the functional $\widetilde{g}$ in~\eqref{e:g*} and the definition in~\eqref{def:prox} we write
		\begin{align*}
		    \prox_{\widetilde{g}}(w) &= \argmin_{y\in L^2([0,1];\dR)}  \widetilde{g}(y) + \frac{1}{2}\,\int_0^1 (y(t) - w(t))^2\, dt \\
            &= \argmin_{y\in L^2([0,1];\dR)}  \iota_{S}(y) + x_f^T\,p(1) - x_0^T\,p(0) + \frac{1}{2} \int_0^1 (y(t) - w(t))^2\, dt\,.
        \end{align*}
		So, by using the definition of $S$ in \eqref{eq:sperp}, finding $\prox_{\widetilde{g}}(w)$ becomes finding $y$ that solves the problem
		\[
		\mbox{(Prox\~{g}) }\left\{\begin{array}{rl}
			\ds\min_{y(\cdot)} & \ \ x_f^T\,p(1) - x_0^T\,p(0) + 
			\ds\frac{1}{2}\int_0^1 \left(y(t) - w(t)\right)^2\, dt \\[5mm]
			\mbox{subject to} & \ \ \dot{p}(t) = -A^T(t)\,p(t)\,,\ \ y(t) = b^T(t)\,p(t)\,, \mbox{ for all\ } t\in[0,1]\,.
		\end{array} \right.
		\]
		Hence, 
		\[
  \begin{array}{rcl}
     \prox_{\widetilde{g}}(w)(t)   & 	=  &y(t) = b^T(t)\,p(t)= w(t) - (-b^T(t)\,p(t) + w(t)) \\
     &&\\
       & =&  w(t) - \prox_g(w)(t)=\prox_{g^*}(w)(t)\,,
  \end{array}
		\]
		where $\dot{p}(t) = -A^T(t)\,p(t)$, and Lemma~\ref{lem:proxfg} was used with $\lambda$ replaced by $p$ in the second to last equality. In the last equality, we used  \cite[Equation~(24.4)]{BC2017}. Therefore,
	we deduce that
		$\prox_{\widetilde{g}}=\prox_{g^*}$
		and by~\cite[Corollary~24.7]{BC2017} we obtain
		\begin{equation}
			\label{e:ft:gt}
			g^* = \widetilde{g} + c, 
		\end{equation}
		where $c$ is a constant. 
		Consequently, we deduce that $\dom g^* =\dom \widetilde{g} $. Combining this fact with \eqref{eq:f*g*} then \eqref{e:g*} yields
		\begin{equation}
			\label{e:S:A-A}
			S=(\cal A-\cal A)^\perp.	
		\end{equation}
	\eqref{t:calc:g^*:iii}:
		It follows from  \eqref{e:ft:gt}, \eqref{eq:aperp3} and \eqref{e:S:A-A}  that
		\begin{equation}
			\label{e:g*:c2}
			g^*=\widetilde{g}+c=	\iota_{(\cal A-\cal A)^\perp}+\scal{\cdot}{a^\sperp}+c.
		\end{equation}	  
We claim that $c  =0$.
		Indeed, 
		since $0=g^*(0)=\widetilde{g}(0)+c=c$. Hence our claim holds.
  This completes the proof.
	\end{proof}		

As a consequence, we have the following proposition which verifies that
 (D) is the Fenchel dual of (P).
	
	\begin{proposition}
		The dual problem~(D) is the Fenchel dual of Problem~(P).
	\end{proposition}
	\begin{proof}
		The statement follows directly by combining \eqref{eq:dualproblem:min:norm},
		Lemma~\ref{lem:conj:f} and 
		Theorem~\ref{t:calc:g^*}.
	\end{proof}

\section{Douglas--Rachford Algorithm}

In this section, we introduce the Douglas--Rachford operator and derive its fixed point for the optimal control problem.  We also present the DR algorithm generating both the primal and dual sequences for an optimal control problem.

	\subsection{Douglas--Rachford operator and its fixed point}
	\label{sec:DR_fixed_pt}
	
	Let $u\in L^2(0,1;\dR)$.
	The Douglas--Rachford operator associated with
	the ordered pair $(f,g)$ is defined by
	\begin{equation}
		\label{def:genDR:fg}
		Tu=u -\prox_fu +\prox_g(2\prox_fu-u).   
	\end{equation}
	Set $\gamma := 1/(1+r)$.  Then
	it follows from, e.g., 
	\cite[Example~24.8(i)~and~Example~23.4]{BC2017}
	that $\prox_f=P_{\cal{B}}\circ (\gamma\Id)$
	and  $\prox_g=P_{\cal{A}}$.
	Therefore, \eqref{def:genDR:fg} becomes
	\begin{equation}
		\label{def:DR:AB:Id}
		Tu=u -P_{\cal{B}}(\gamma u)
			+P_{\cal{A}}\Big(2P_{\cal{B}}(\gamma u)-u\Big). 
	\end{equation}
	Observe that under appropriate constraint qualifications,
	e.g., $\cal{A}\cap\intr \cal{B}\not=\varnothing$, it is well-known that
	solving \eqref{eq:problem:min:norm} is equivalent to solving 
	the inclusion:
	\begin{equation}
		\label{e:inc:sd}
		\text{Find $u\in L^2(0,1;\dR)$\;\;\; such that\;\;\; }
		0\in u+N_{\cal{B}}u+N_{\cal{A}}u
		=u+N_{\cal{B}}u+(\cal{A}-\cal{A})^{\perp}.
	\end{equation}
	Similarly, solving 	\eqref{eq:dualproblem:min:norm} is equivalent to solving
	the Attouch--Th\'era dual problem of \eqref{e:inc:sd}; namely
	\begin{equation}
		\label{e:inc:dual}
		\text{Find $w\in L^2(0,1;\dR)$\;such that\; }
		0\in a^\sperp+P_{\cal{B}}(-w)+N_{(\cal{A}-\cal{A})^\perp}w
		=a^\sperp+P_{\cal{B}}(-w)+ \cal{A}-\cal{A}.
	\end{equation}
	Set
	\begin{equation}
		\label{e:p:sol}
		Z:=\menge{u\in L^2(0,1;\dR)}{0\in u+N_{\cal{B}}u+(\cal{A}-\cal{A})^{\perp}},
	\end{equation}	
	and set
	\begin{equation}
		\label{e:d:sol}
		K:=\menge{w\in L^2(0,1;\dR)}{0\in a^\sperp+P_{\cal{B}}(-w)+ \cal{A}-\cal{A}}.
	\end{equation}	
	We use $\fix T$ to denote the fixed point set of $T$ defined by $\fix T=\menge{x\in L^2(0,1;\dR)}{x=Tx}$.
	
	Interestingly, the fixed point set of $T$
	can be expressed using the sets $Z$ and $K$. 
	This is summarized in the following fact.
	
	\begin{fact}  \label{fact:prim+dual}
		Let $T$ be the Douglas--Rachford operator 
		defined in \eqref{def:DR:AB:Id}.
		Then
		\begin{equation}
  \label{eq:fixT:ZK}
			\fix T=Z+K.	
		\end{equation}	
	\end{fact}		
	\begin{proof}
		See	\cite[Corollary~5.5(iii)]{BauBotHarMou2012}.
	\end{proof}	 

The following fact provides a sufficient condition for the existence of a fixed point of the DR operator.
\begin{fact}
\label{fact:ZK:nonemp}
Let $f$ and $g$ be defined as in \eqref{eq:fg}.
Then $f$ is strongly convex.
Suppose that $\cal{A}\cap\intr \cal{B}\not=\varnothing$.
Then
\begin{equation}
\label{eq:ZK:nonemp}
Z\neq \varnothing 
\iff K\neq \varnothing 
\iff \fix T \neq \varnothing.
\end{equation}
Moreover, if $Z\neq \varnothing $ 
then $Z$ is a singleton.
 \end{fact}
 \begin{proof}
The claim that $f$ is strongly convex is clear.
 We now turn to \eqref{eq:ZK:nonemp}.
The first equivalence 
follows from combining \cite[Corollary~16.48(ii)]{BC2017}
and \cite[Corollary~3.2]{AT} 
or \cite[Theorem~7.1]{BauHarMou2014}.
The second equivalence
is a direct consequence of \eqref{eq:fixT:ZK}
 and the first equivalence.
Finally, suppose that $Z\neq \varnothing$.
Because 
 $f$ is strongly convex,
 the result now follows from, 
 e.g., \cite[Corollary~28.3(v)]{BC2017}.
\end{proof}

		

Fact \ref{fact:prim+dual} describes the structure of the set of fixed points of the Douglas--Rachford operator. Together with Fact \ref{fact:ZK:nonemp}, these two results imply that the sum of a primal solution and a dual solution produce a fixed point of $T$, as long as a primal solution exists. The following theorem provides the particular structure of $\fix T$ for the case of Problem~(P).  It also reconfirms the result in Fact~\ref{fact:prim+dual}.
 
	\begin{theorem}  \label{thm:fixedpt}
		If $\varphi$ is a fixed point of $T$, then
		\begin{equation}  \label{exp:varphi}
			\varphi(t) = 
			u(t) + w(t) = 
			\left\{\begin{array}{ll}
				\overline{a} - b^T(t)\,\lambda(t)\,, & \mbox{ if\ \ } -b^T(t)\,\lambda(t) > r\,\overline{a}\,, \\[2mm]
				-(1+r)\,b^T(t)\,\lambda(t)/r\,, & \mbox{ if\ \ } r\,\underline{a} \le -b^T(t)\,\lambda(t)
				\le r\,\overline{a}\,, \\[2mm]
				\underline{a} - b^T(t)\,\lambda(t)\,, & \mbox{ if\ \ } -b^T(t)\,\lambda(t) < r\,\underline{a}\,.
			\end{array}\right.
		\end{equation}
		where $u$ is the (unique) solution of the primal problem~{\em (P)} and $w$ is a solution of the dual problem~{\em (D)}. 
	\end{theorem}
	\begin{proof}
		Suppose that $\varphi$ is a fixed point of $T$.  Then $T\varphi = \varphi$ and \eqref{def:DR:AB:Id} can be re-written as
		\[
		{\cal P}_{\cal{B}}(\gamma \varphi) = {\cal P}_{\cal{A}}\Big(2{\cal P}_{\cal{B}}(\gamma \varphi)-\varphi\Big)\,.
		\]
		In other words, the problem is one of finding the fixed point $\varphi$ which solves the system of equations
		\begin{eqnarray}
			{\cal P}_{\cal{A}}\Big(2{\cal P}_{\cal{B}}(\gamma \varphi)-\varphi\Big) &=& \beta\,, \label{eq:fp1} \\[1mm]
			{\cal P}_{\cal{B}}(\gamma \varphi) &=& \beta\,.  \label{eq:fp2}
		\end{eqnarray}
		Note we can rewrite Equation \eqref{eq:fp1} using \eqref{eq:fp2} as ${\cal P}_{\cal{A}}(2\beta-\varphi) = \beta$.  From \eqref{eq:proxg} in Lemma~\ref{lem:proxfg} and Remark~\ref{rem:projA} we have ${\cal P}_{\mathcal{A}}(u)(t) = u(t) - b^T(t)\,\lambda(t)$, with $\dot{\lambda}(t) = -A^T(t)\,\lambda(t)$.  Therefore,
		\[
		{\cal P}_{\cal{A}}(2\beta-\varphi)(t) = (2\,\beta(t) - \varphi(t)) - b^T(t)\,\lambda(t) = \beta(t)\,,
		\]
		and, re-arranging,
		\begin{equation}  \label{exp:f}
			\varphi(t) = \beta(t) - b^T(t)\,\lambda(t)\,.
		\end{equation}
		It is straightforward to write down the projection in Equation~\eqref{eq:fp2} onto the box ${\cal B}$ as
		\[
		\beta(t) = \left\{\begin{array}{ll}
			\overline{a}\,, & \mbox{ if\ \ } \gamma\,\varphi(t) > \overline{a}\,, \\[2mm]
			\gamma\,\varphi(t)\,, & \mbox{ if\ \ } \underline{a} \le \gamma\,\varphi(t)
			\le \overline{a}\,, \\[2mm]
			\underline{a}\,, & \mbox{ if\ \ } \gamma\,\varphi(t) < \underline{a}\,.
		\end{array}\right.
		\]
		Using $\gamma = 1/(1+r)$, $\beta$ becomes
		\[
		\beta(t) = \left\{\begin{array}{ll}
			\overline{a}\,, & \mbox{ if\ \ } \varphi(t) > (1+r)\,\overline{a}\,, \\[2mm]
			\varphi(t)/(1+r)\,, & \mbox{ if\ \ } (1+r)\,\underline{a} \le \varphi(t)
			\le (1+r)\,\overline{a}\,, \\[2mm]
			\underline{a}\,, & \mbox{ if\ \ } \varphi(t) < (1+r)\,\underline{a}\,.
		\end{array}\right.
		\]
		Substituting this into~\eqref{exp:f},
		\[
		\varphi(t) = \left\{\begin{array}{ll}
			\overline{a} - b^T(t)\,\lambda(t)\,, & \mbox{ if\ \ } \varphi(t) > (1+r)\,\overline{a}\,, \\[2mm]
			\varphi(t)/(1+r) - b^T(t)\,\lambda(t)\,, & \mbox{ if\ \ } (1+r)\,\underline{a} \le \varphi(t)
			\le (1+r)\,\overline{a}\,, \\[2mm]
			\underline{a} - b^T(t)\,\lambda(t)\,, & \mbox{ if\ \ } \varphi(t) < (1+r)\,\underline{a}\,.
		\end{array}\right.
		\]
		Solving for $\varphi$ in the second equation above, we derive
		\begin{equation*}
			\varphi(t) = \left\{\begin{array}{ll}
				\overline{a} - b^T(t)\,\lambda(t)\,, & \mbox{ if\ \ } \varphi(t) > (1+r)\,\overline{a}\,, \\[2mm]
				-b^T(t)\,\lambda(t)/r - b^T(t)\,\lambda(t)\,, & \mbox{ if\ \ } (1+r)\,\underline{a} \le \varphi(t)
				\le (1+r)\,\overline{a}\,, \\[2mm]
				\underline{a} - b^T(t)\,\lambda(t)\,, & \mbox{ if\ \ } \varphi(t) < (1+r)\,\underline{a}\,,
			\end{array}\right.
		\end{equation*}
		and substituting $\varphi(t)$ in the domain expressions,
		\begin{equation*}
			\varphi(t) = \left\{\begin{array}{ll}
				\overline{a} - b^T(t)\,\lambda(t)\,, & \mbox{ if\ \ } -b^T(t)\,\lambda(t) > r\,\overline{a}\,, \\[2mm]
				-b^T(t)\,\lambda(t)/r - b^T(t)\,\lambda(t)\,, & \mbox{ if\ \ } r\,\underline{a} \le -b^T(t)\,\lambda(t)
				\le r\,\overline{a}\,, \\[2mm]
				\underline{a} - b^T(t)\,\lambda(t)\,, & \mbox{ if\ \ } -b^T(t)\,\lambda(t) < r\,\underline{a}\,,
			\end{array}\right.
		\end{equation*}
		where $\dot{\lambda}(t) = -A^T(t)\,\lambda(t)$.  Then a minor manipulation yields the right-hand side expression in~\eqref{exp:varphi}.   The fact that $p = -\lambda$ and $w(t) = -b^T(t)\,\lambda(t)$ facilitates $\varphi(t) = u(t) + w(t)$, $u$ the (unique) solution of the primal problem~(P) and $w$ a solution of the dual problem~(D), as the left-hand side expression in~\eqref{exp:varphi}.
	\end{proof}

	\subsection{The Algorithm}
	\label{sec:algo}
	Suppose that $\cal{A}\cap\intr \cal{B}\not=\varnothing$.	
The DR operator in~\eqref{def:DR:AB:Id} can be employed in an algorithm with clear steps (see Fact~\eqref{fact:DR:conv} below).  Each time the operator is applied it results in the primal iterate (update) in Step~5.  The dual iterate (update) 
is given in Step~4.

	\begin{algorithm} \rm {\bf (Douglas--Rachford)} \label{algo:DR_primal_dual}\ \rm
		\begin{description}
			\item[Step 1] ({\em Initialization}) Choose a parameter $\gamma\in\left(0,1\right)$ and the initial iterate $u_0$ arbitrarily. 
			Choose a small parameter $\varepsilon>0$, and set $k=0$. 
			\item[Step 2] ({\em Projection onto ${\cal B}$})  Set $u^- = \gamma u_{k}$. 
			Compute $\widetilde{u} = P_{{\cal B}}(u^-)$. 
			\item[Step 3] ({\em Projection onto ${\cal A}$}) Set $u^- := 2\widetilde{u}-u_k$. 
			Compute $\widehat{u} = P_{{\cal A}}(u^-)$.
			\item[Step 4] ({\em Dual update}) $w_k := u_k - \widetilde{u}$.
			\item[Step 5] ({\em Primal update}) Set $u_{k+1} := w_k + \widehat{u}$.
			\item[Step 6] ({\em Stopping criterion}) If $\|u_{k+1} - u_k\|_{L^\infty} \le \varepsilon$, then RETURN $\widetilde{u}$ and STOP.  
			Otherwise, set $k := k+1$ and go to Step 2.
		\end{description}
	\end{algorithm} 
 
 The following fact establishes strong convergence of the primal iterates and weak convergence of the dual iterates in Algorithm~\ref{algo:DR_primal_dual}.
 \begin{fact}
 \label{fact:DR:conv}
Suppose that $\cal{A}\cap\intr \cal{B}\not=\varnothing$.
Let 
$Z$ and $K$ be defined as in 
\eqref{e:p:sol} and \eqref{e:d:sol}
respectively.
Let $(u_k)_{k\in \mathbb{N}}$
 and $(w_k)_{k\in \mathbb{N}}$
 be defined as in Algorithm~\ref{algo:DR_primal_dual}.
 Then  $(\exists \overline{u},\,\overline{w} \in L^2([0,1];\dR))$
such that $Z=\{\overline{u}\}$,
$\overline{w}\in K$,
$u_k\to \overline{u}$,
 and $w_k\weakly \overline{w}$.
 \end{fact}
\begin{proof}
Combine 
Fact~\ref{fact:ZK:nonemp} and, e.g., \cite[Theorem~28.3(iv)\&(v)(b)]{BC2017}.
\end{proof}

\begin{remark} \rm
The iterates in Algorithm~\ref{algo:DR_primal_dual} are functions, and in a numerical implementation of the algorithm, it is not possible to perform addition or scalar multiplication of functions.  Therefore these function iterates are represented by their discrete approximations, where each iterate, for example, $u_k$, is given as a vector in $\dR^N$, with components $u_{k,i} \approx u_k(t_i)$, and $t_{i+1} = t_i + i\,h$, $t_0 = 0$, $h = 1/N$, $i = 0,1,\ldots,N-1$.  One should note however that this is different from the direct discretization of the problem, where the problem itself is discretized and the discretized problem is dealt with as a finite-dimensional one, instead of the infinite-dimensional one here. In Step~6 of the algorithm the $L^\infty$ norm is used to measure the closeness of two consecutive iterates as the $L^\infty$ norm is a better (or realistic) norm to use for example than the $L^2$ norm in measuring the closeness of functions.
\proofbox
\end{remark}

\begin{remark} \rm
Controllability of the dynamical system subject to the bounds on the control variable ensures that a solution exists from any point to any other point in the state space.  This implies that the intersection of the sets $\cal{A}$ and the interior of $\cal{B}$ is nonempty.  In the case when the bounds on the control are too restricted, this intersection will be empty, which is also a realistic situation encountered in engineering problems, for example the motor torque being too small to achieve a target state, which requires a special treatment in finding a control approximate in some sense.  The latter (i.e., the infeasible case) is a subject of future investigation.
\proofbox
\end{remark}

\section{Numerical Experiments}

\subsection{Double integrator}

The double integrator is modelled as a special instance of~\eqref{ODE} with
\[
A = \left[\begin{array}{cc}
0 &\ 1 \\
0 &\ 0
\end{array}\right],\quad
b = \left[\begin{array}{c}
0 \\
1
\end{array}\right],
\]
$x_0 = (s_0, v_0)$ and $x_f = (s_f, v_f)$. Although Problem~(P) appears in a simple form with these choices for $A$ and $b$, since the control variable is constrained, it is not possible to get a solution analytically. Indeed,  numerical methods, such as the one we discuss in this paper, are needed to obtain an approximate solution.  The double integrator problem is simple and yet rich enough to study when introducing and illustrating many basic and new concepts or when testing new numerical approaches in optimal control---see, for example~\cite{BauBurKay2019, Kaya2020}, including the text book~\cite{Locatelli2017}.

The following two facts provide the projectors onto ${\cal A}$ and ${\cal B}$ for the minimum-energy control of the double integrator (see \cite[Proposition 2.1 and Proposition 2.2]{BauBurKay2019}).

\begin{fact}[Projection onto ${\cal A}$]  \label{cor:projA_PDI}
The projection $P_{\cal A}$ of $u^-\in {L}^2([0,1];\dR)$ onto the constraint set ${\cal A}$ is given by
\[
P_{{\cal A}}(u^-)(t) = u^-(t) + c_1\,t + c_2\,,
\]
for all $t\in[0,1]$, where
{\begin{eqnarray*}
&& c_1 := 12\left(s_0+v_0-s_f+\int_0^1 (1-\tau)u^-(\tau)d\tau\right)-6\left(v_0-v_f+\int_0^1 u^-(\tau)d\tau\right), \\[1mm]
&& c_2 := -6\left(s_0+v_0-s_f+\int_0^1 (1-\tau)u^-(\tau)d\tau\right)+2\left(v_0-v_f+\int_0^1 u^-(\tau)d\tau\right).
\end{eqnarray*}}
\end{fact}
We note that the integrals appearing in the expressions for $c_1$ and $c_2$ can only be evaluated numerically.

\begin{fact}[Projection onto \boldmath{${\cal B}$}]  \label{cor:projB}
The projection $P_{{\cal B}}$ of $u^-\in {L}^2([0,t_f];\dR)$ onto the constraint set ${\cal B}$ is given by
\begin{equation*}
P_{{\cal B}}(u^-)(t) = \left\{\begin{array}{rl}
\overline{a}\,, &\ \ \mbox{if\ \ } u^-(t)\ge \overline{a}\,, \\[1mm]
u^-(t)\,, &\ \ \mbox{if\ \ } \underline{a}\le u^-(t)\le \overline{a}\,, \\[1mm]
\underline{a}\,, &\ \ \mbox{if\ \ } u^-(t)\le \underline{a}\,, \\[1mm]
\end{array} \right.
\end{equation*}
for all $t\in[0,1]$. 
\end{fact}

	Using $w = b^T(t)\,p(t) = -b^T(t)\,\lambda(t)$ in \eqref{opt_u(t)}, one can re-write the optimal control $u$ (the primal variable) in terms of the dual variable $w$ for the double integrator problem as	
	\begin{equation}  \label{eq:opt_u_&_w}
		u(t) = \left\{\begin{array}{ll}
			\overline{a}\,, &\ \ \mbox{if\ \ } w(t) > r\,\overline{a}\,, \\[1mm]
			w(t)/r\,, &\ \ \mbox{if\ \ } r\,\underline{a} \le w(t) \le r\,\overline{a}\,, \\[1mm]
			\underline{a}\,, &\ \ \mbox{if\ \ } w(t) < r\,\underline{a}\,;
		\end{array} \right.
	\end{equation}
	in other words, $u = \mathcal{P}_{\mathcal{B}}(w/r)$.
	
 Consider the case (as in~\cite{BauBurKay2019}) when $s_0 = s_f = v_f = 0$, $v_0 = 1$, $\underline{a} = -2.5$, $\overline{a} = 2.5$, $r = 1/3$ (i.e., $\gamma = 0.75$). From empirical evidence in Figure~1.3(a) in \cite{BauBurKay2019} $\gamma = 0.75$ provides optimal performance for this problem.  Application of Algorithm~\ref{algo:DR_primal_dual} yields the graphical solutions as shown in Figure~\ref{fig:PDI_duality}.  With $u_0 = 0$ and $\varepsilon = 10^{-8}$, the algorithm has converged in 127 iterations.  The solid (blue) plot in Figure~\ref{fig:PDI_duality} is the optimal primal control solution denoted by $u$ and the dotted (yellow) plot is the fixed point of the DR operator denoted by $\varphi$.  Recall by Theorem~\ref{thm:fixedpt} that $\varphi = u + w$, which is reconfirmed by the solution curves in Figure~\ref{fig:PDI_duality}.  Finally, a close inspection of the plots reveals that the optimality condition~\eqref{eq:opt_u_&_w} is verified, reconfirming the optimality of $u$.  This check is something that was not possible to do in~\cite{BauBurKay2019}.
 
    \begin{figure}[t!]
		\centering
		\includegraphics[width=100mm]{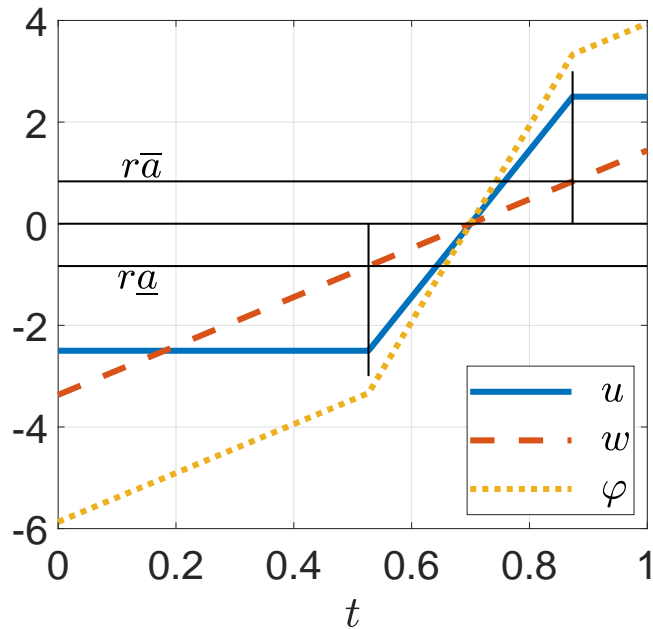}
		\caption{Double integrator---The solid (blue) plot is the optimal primal control solution $u$, the dashed (orange) plot is the dual solution $w$ and the dotted (yellow) plot is the fixed point $\varphi$ of the DR operator.}
		\label{fig:PDI_duality}
	\end{figure}

We note that the optimal primal solution $u$ is unique, no matter what the value of $r\in(0,\infty)$, or $\gamma=1/(1+r)\in(0,1)$, is, of course
(see Fact~\ref{fact:ZK:nonemp}).  On the other hand, the optimal dual solution $w$ depends on the parameter $r$: as $r$ changes the slope of the line representing the graph of $w$ changes with the $t$-intercept remaining the same.  The fixed point $\varphi$ also evolves to obey $\varphi =  u + w$. For a fixed $r$ we cannot prove that $w$ is unique but after running Algorithm~\ref{algo:DR_primal_dual} with $r=1/3$ using $100,000$ randomly generated starting points we always converged to the same $w$.

\subsection{Machine tool manipulator}
A machine tool manipulator is a machine used to simulate human hand movements usually found in manufacturing. In \cite{ChrMauZir2010,BurKayMaj2014} this is modelled as an optimal control problem of the form in \eqref{ODE} with
\begin{align*}
A &= \begin{bmatrix}
0 & 0 & 0 & 1 & 0 & 0 & 0 \\
0 & 0 & 0 & 0 & 1 & 0 & 0 \\
0 & 0 & 0 & 0 & 0 & 1 & 0 \\
-4.441\times10^7/450 & 0 & 0 & -8500/450 & 0 & 0 & -1/450 \\
0 & 0 & 0 & 0 & 0 & 0 & 1/750 \\
0 & 0 & -8.2\times10^6/40 & 0 & 0 & -1800/40 & 0.25/40 \\
0 & 0 & 0 & 0 & 0 & 0 & -1/0.0025
\end{bmatrix}, \\ b &= \begin{bmatrix}
0 & 0 & 0 & 0 & 0 & 0 & 1/0.0025
\end{bmatrix}^T,
\end{align*}
$x_0=(0,0,0,0,0,0,0)$, $x_f=(0,2.7\times10^{-3},0,0,0.1,0,0)$, $t_f=0.522$. We alter the objective function given in the previously mentioned references to be the minimum-energy control as in \cite{BurCalKay2022}. As with the double integrator problem we cannot solve this problem analytically hence the need for a numerical method such as the one in this paper. The projector onto $\cal{B}$ is as given in Fact \ref{cor:projB} with $t_f=0.0522$. Due to the large number of state variables involved in this problem it is not feasible to give the projector onto $\cal{A}$ in closed form so we implement a numerical approach to approximate the projector. The numerical algorithm and expression for the projector onto $\cal{A}$ are given in \cite{BurCalKay2022}.

\begin{figure}[t!]
		\centering
		\includegraphics[width=100mm]{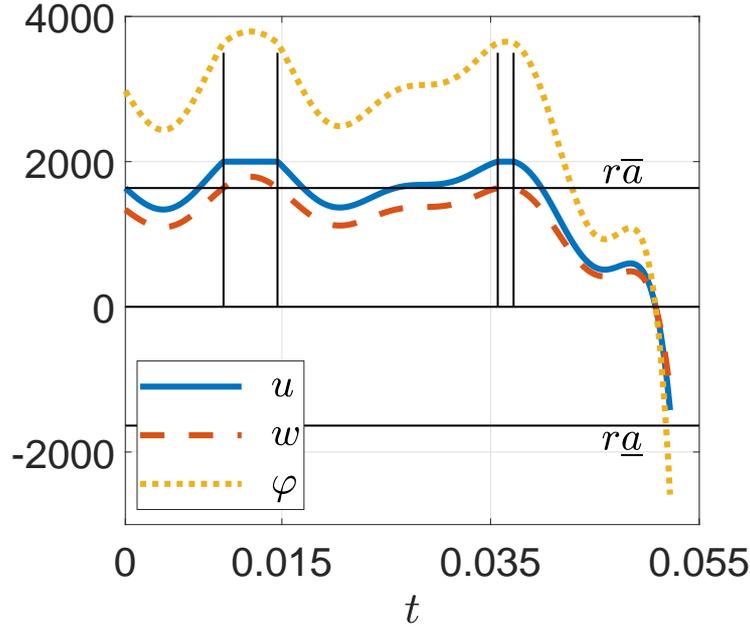}
		\caption{Machine tool manipulator---The solid (blue) plot is the optimal primal control solution $u$, the dashed (orange) plot is the dual solution $w$ and the dotted (yellow) plot is the fixed point $\varphi$ of the DR operator.}
		\label{fig:MTM_duality}
	\end{figure}

For this problem let $\underline{a} = -2000$, $\overline{a} = 2000$, $r=1/0.55-1$ (i.e., $\gamma=0.55$). The value of $r$ will not impact the solution but from experiments in \cite{BurCalKay2022}, $\gamma=0.55$ had the fastest performance. With $u_0=0$ and $\varepsilon=10^{-4}$ Algorithm~\ref{algo:DR_primal_dual} converged in 249 iterations. From Figure~\ref{fig:MTM_duality} we can again observe that, as stated in Theorem~\ref{thm:fixedpt}, $\varphi=u+w$. We can also see that Condition~\eqref{eq:opt_u_&_w} is verified by the plots in Figure~\ref{fig:MTM_duality} which once more confirms the optimality of $u$.


\section{Conclusion}

We have explored relationships between the primal and dual optimal control problems as the DR algorithm is applied to solve them.  We derived the Fenchel dual to the primal problem.  We provided an explicit expression 
of the set of fixed points of the DR operator as the Minkowski sum of the sets  of primal and dual solutions.  
We showed that the fixed point expression can be used as a certificate of optimality in that the optimality conditions for optimal control obtained numerically can be checked.

As an example, we first chose the minimum-energy control of the double integrator, which is simple yet rich enough to illustrate the concepts we developed and the results we obtained. We also applied our methodology to the minimum-energy control of a machine tool manipulator model which is numerically more challenging to solve than the double integrator.

In the future, the work we did here should be extended to general LQ optimal control problems, also involving constraints on the state variables.  We considered only constraints on the control variables in the present paper.  The inclusion of state variable constraints is well known to create theoretical and numerical challenges.  The derivation of the dual of LQ control problems, albeit in a space different to that of the primal, done in~\cite{AltKaySch2016,BurKayMaj2014} involves only constraints on the control variables, so the work there needs to be extended first.

\section*{Acknowledgments}

The authors are indebted to the two anonymous reviewers whose comments and suggestions improved the paper.  The research of BIC is supported by an Australian Government Research Training Program Scholarship, as well as by the funding provided by the Universities of South Australia and Waterloo for her three-month visit to the University of Waterloo.
The research of WMM is supported by the Natural Sciences and Engineering Research Council of Canada Discovery Grant and the Ontario Early Researcher Award.

\end{document}